\documentclass[reqno,12pt]{amsart}

\usepackage{epsf}
\usepackage{graphics}
\usepackage{amssymb}
\usepackage{amsmath}

\date{}

\theoremstyle{plain}
\newtheorem{theorem}{Theorem}
\newtheorem{corollary}{Corollary}

\newtheorem{proposition}{Proposition}

\theoremstyle{definition}
\newtheorem{definition}{Definition}

\theoremstyle{remark}
\newtheorem*{example}{Example}

\def\C{{\mathbb C}}
\def\N{{\mathbb N}}

\def\R{{\mathbb R}}

\title{Bipartite graphs and combinatorial adjacency} 

\author{Sebastian Baader}

\begin{document}

\begin{abstract} We present a simple combinatorial model for quasipositive surfaces and positive braids, based on embedded bipartite graphs. As a first application, we extend the well-known duality on standard diagrams of torus links to twisted torus links. We then introduce a combinatorial notion of adjacency for bipartite graph links and discuss its potential relation with the adjacency problem for plane curve singularities.
\end{abstract}

\maketitle

\section{Introduction}

The main symmetry of torus links, $T(p,q)=T(q,p)$, is a trivial geometric fact. However, it is hardly visible on the level of standard diagrams; the braids $(\sigma_1 \sigma_2 \ldots \sigma_{p-1})^q$ and $(\sigma_1 \sigma_2 \ldots \sigma_{q-1})^p$ do not even have the same number of crossings. We propose a diagram of the fibre surface of torus links that exhibits the symmetry of the parameters $p$ and $q$. The following description is motivated by A'Campo's new t\^ete-\`a-t\^ete vision of the monodromy of isolated plane curve singularities~\cite{AC}. It relies on the fact that the fibre surface of the torus link $T(p,q)$ retracts on a complete bipartite graph of type $\theta_{p,q}$~\cite{Ph}. This fact can be seen explicitly in Figure~1, where the fibre surface of the torus knot $T(3,4)$ is drawn as a union of $12$ ribbons along the edges of the graph $\theta_{3,4}$ embedded in $\R^3$. We will shortly explain this in detail. 

\begin{figure}[ht]
\scalebox{0.7}{\raisebox{-0pt}{$\vcenter{\hbox{\epsffile{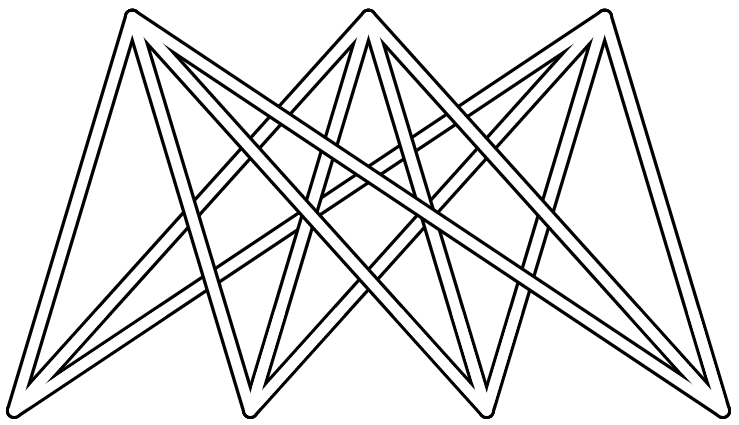}}}$}} 
\caption{}
\end{figure}

Ribbon diagrams offer a lot of possibilities to perform cobordisms by cutting ribbons. For example, the three cuts shown in Figure~2 result in a disjoint union of two trefoil knots. It turns out that many ribbon cuts correspond to smoothings of certain crossing in the standard diagrams of torus links.
This suggests to look at links associated with subgraphs of the complete bipartite graphs $\theta_{p,q} \subset \R^3$. Let us call these bipartite graph links. 

\begin{figure}[ht]
\scalebox{0.7}{\raisebox{-0pt}{$\vcenter{\hbox{\epsffile{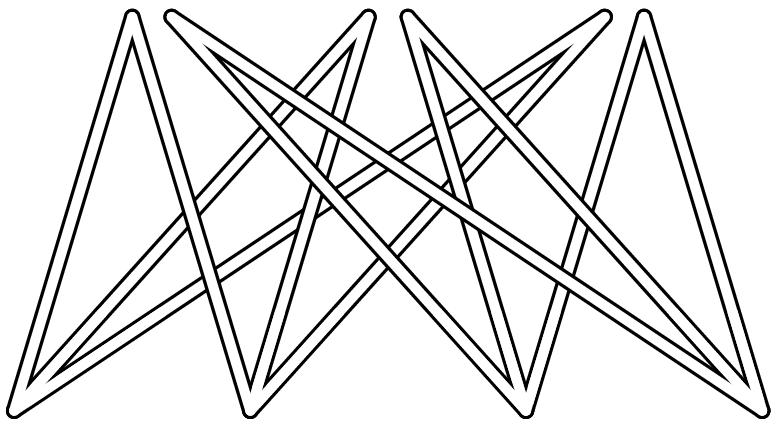}}}$}} 
\caption{}
\end{figure}

\begin{theorem} The family of bipartite graph links coincides with the family of strongly quasipositive links.
\end{theorem}

Bipartite graph links contain various well-studied classes of links, for example positive braids links and Lorenz links. We will single out these classes in terms of bipartite graphs. As an application, we obtain a curious duality on standard diagrams of twisted torus links. This is the content of Section~3.

Given two fixed natural numbers $p,q \geq 2$, we may ask which links can be obtained by cutting ribbons of the embedded complete bipartite graph $\theta_{p,q}$. This innocent question turns out to be a hard one. In fact, it may be related to the adjacency problem for plane curve singularities. We will offer perspectives on this in Section~4.

The last section is devoted to a notion of density that comes naturally with bipartite graph links. As we will see, links with a high density share at least one property with torus links: their signature invariant has a high defect from maximality.

\section*{Acknowledgements}
This manuscript was greatly influenced by various people. Special thanks go to Norbert A'Campo, Peter Feller, Christian Graf and Masaharu Ishikawa for their inspiring inputs.

\section{Ribbon diagrams for strongly quasipositive links}

We need precise definitions for bipartite graph links, quasipositive surfaces and strongly quasipositive links before proving Theorem~1.
Let $U,L \subset \R^3$ be two skew lines. We will fix $U=\{x=0,\, z=1\}$, $L=\{x=y,\, z=0\}$, for simplicity. Let $\Gamma \subset \R^3$ be a finite union of straight line segments, each one having one endpoint on $U$ and one on $L$. Thus $\Gamma$ is an embedded bipartite graph. The ribbon surface associated with $\Gamma$ is made up of ribbons, one for each edge of $\Gamma$, whose projections onto the $(y,z)$-plane are immersions (see Figures~1, 2 and~4 for an illustration).

\begin{definition}  A link in $\R^3$ is called a bipartite graph link, if it is the boundary of a ribbon surface in the above sense.
\end{definition}

We claim that the links associated with complete bipartite graphs are precisely torus links. In fact, given an embedded complete bipartite graph $\Gamma \subset \R^3$ with $p$ and $q$ vertices on $U$ and $L$, respectively, we may deform the corresponding link $L(\Gamma)$ into the standard braid diagram $(\sigma_1 \sigma_2 \ldots \sigma_{p-1})^q$ of $T(p,q)$. This deformation is performed in two steps.

\medskip
(1) Each vertex on the lower line $L$ is adjacent to precisely $p$ edges of $\Gamma$, the union of which we call a fork. Thus $\Gamma$ consists of $q$ forks that are piled in some sense. Stretching each of the vertices on $U$ into an interval allows us to separate all forks in the $(y,z)$-projection. This is shown on the  top left of Figure~3, for the case $p=3$, $q=4$.

\begin{figure}[ht]
\scalebox{0.7}{\raisebox{-0pt}{$\vcenter{\hbox{\epsffile{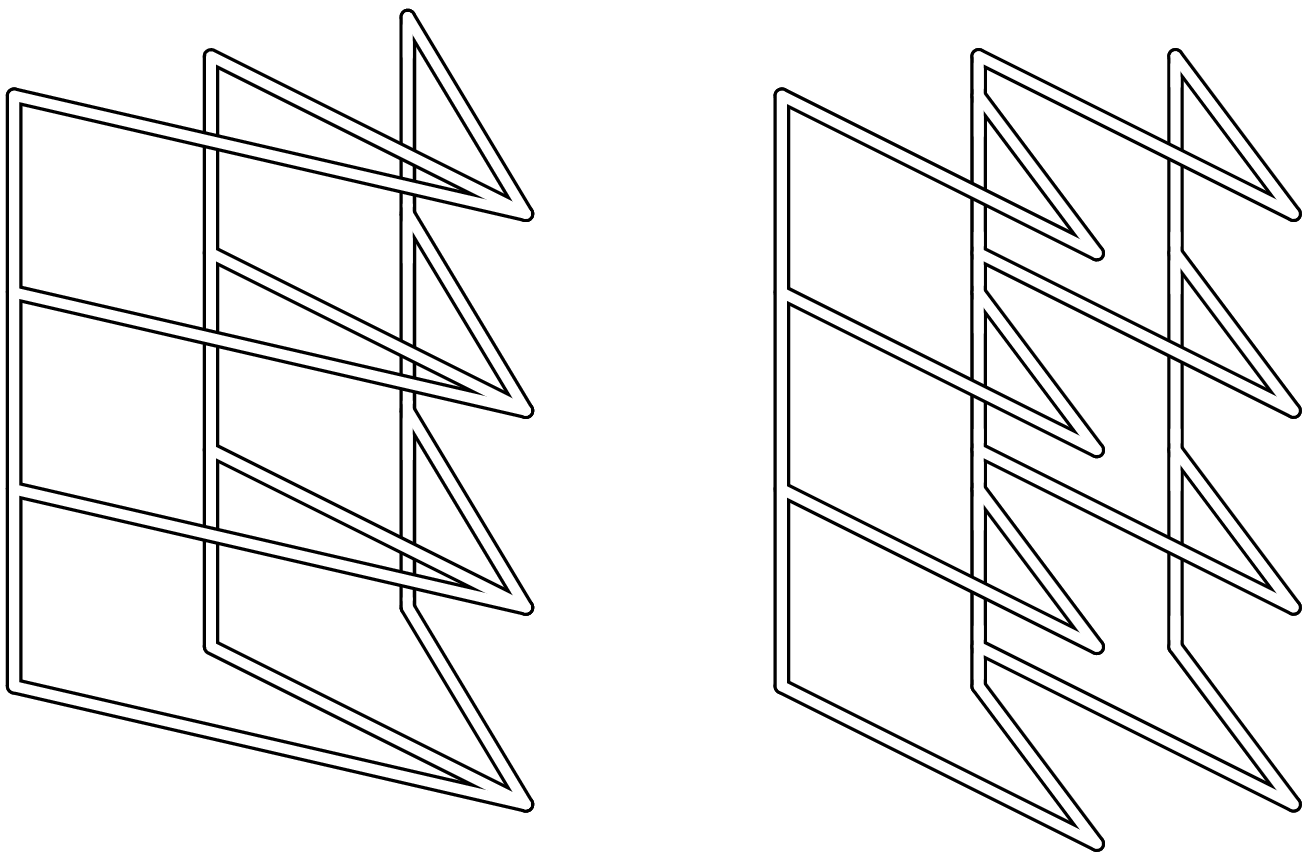}}}$}} 

\bigskip
\bigskip
\scalebox{0.7}{\raisebox{-0pt}{$\vcenter{\hbox{\epsffile{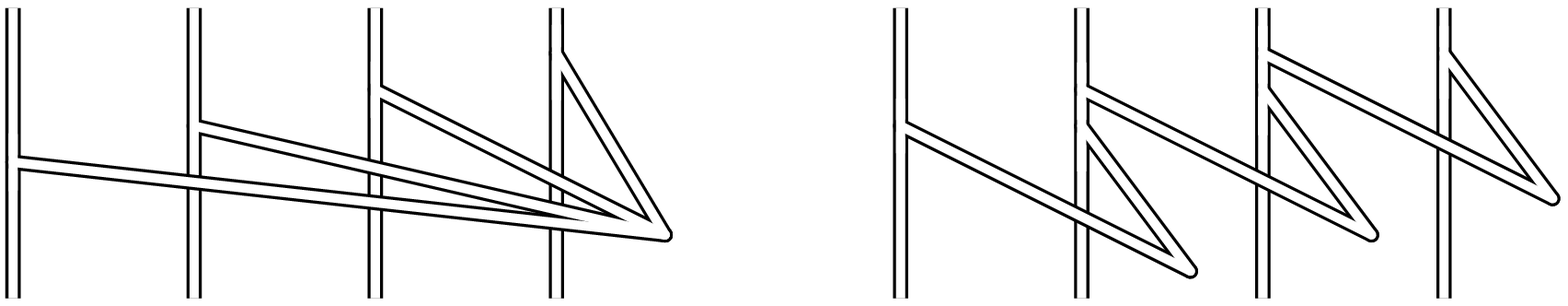}}}$}} 
\caption{}
\end{figure}

(2) Each fork can be split into $p-1$ forks with two teeth connecting a pair of stretched vertices, by a suitable isotopy. This is illustrated at the bottom of Figure~3, for a single fork with $4$ teeth. The resulting surface diagram is shown on the top right of Figure~3 and can easily be identified as the canonical Seifert surface associated with the braid $(\sigma_1 \sigma_2 \ldots \sigma_{p-1})^q$. 

A first naive and false guess is that bipartite graph links are positive braid links. The bipartite graph knot depicted in Figure~4 is obtained from the ribbon diagram of Figure~1 by $4$ ribbon cuts. It is isotopic to the non-fibred positive twist knot $5_2$, which is not a positive braid knot, since these are all fibred~\cite{St}. For the same reason, it is not possible to obtain the knot $5_2$ from any of the standard diagrams of the torus knot $T(3,4)$ by smoothing any number of crossings.

\begin{figure}[ht]
\scalebox{0.7}{\raisebox{-0pt}{$\vcenter{\hbox{\epsffile{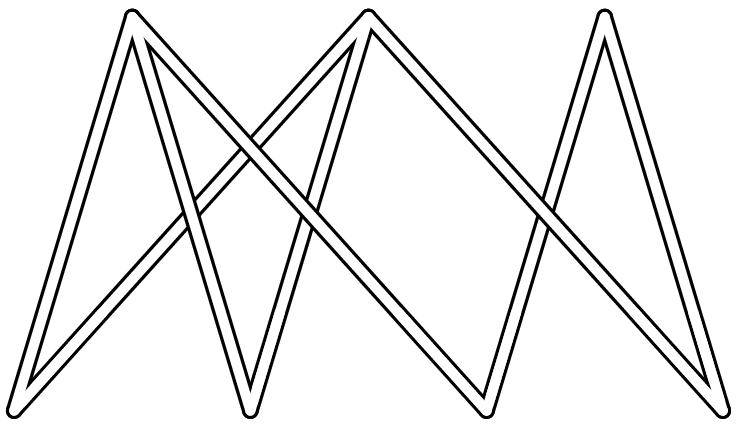}}}$}} 
\caption{}
\end{figure}

There is still a weaker notion of positivity inherited by bipartite graph links: strong quasipositivity. As many notions of positivity, strong quasipositivity was introduced by Lee Rudolph~\cite{Ru1}. The definition of strong quasipositivity involves certain surfaces, altough it is an attribute for links.

\begin{definition} An embedded compact surface $S \subset \R^3$ is called positive, if it is isotopic to an incompressible subsurface of the fibre surface of a positive torus link.
\end{definition}

Here incompressibility simply means that the inclusion of $S$ into the fibre surface induces an injective map on the level of fundamental groups. 

\begin{definition} A link in $\R^3$ is called strongly quasipositive, if it is the boundary of a quasipositive surface.
\end{definition}

Alternatively, quasipositive surfaces may be defined as Legendrian ribbons with respect to the standard contact strucure on $\R^3$~\cite{BI}. As the name suggests, strongly quasipositive links include positive links. However, this is a non-trivial fact, due to Rudolph~\cite{Ru2}.

Using either of the definitions, we immediately see that bipartite graph links are strongly quasipositive. Indeed, all ribbon surfaces constructed above are incompressible subsurfaces of the ribbon surface associated with an embedded complete bipartite graph. The latter are fibre surfaces of torus links, since they are isotopic to the canonical Seifert surfaces of these. 

In order to show the converse, we need yet another desription of strongly quasipositive links, which is in fact the original one~\cite{Ru1}.

\begin{definition} A link in $\R^3$ is strongly quasipositive, if it is the closure of a strongly quasipositive braid $\beta$ in some braid group $B_n$, i.e. a finite product of words of the form
$$\sigma_{i,j}=(\sigma_i \sigma_{i+1} \ldots \sigma_{j-2}) \sigma_{j-1} (\sigma_i \sigma_{i+1} \ldots \sigma_{j-2})^{-1},$$
for $1 \leq i<j \leq n-1$.
\end{definition}

For obvious reasons, we call the words $\sigma_{i,j}$ generalised positive crossings (see Figure~5). Strongly quasipositive links bound canonical Seifert surfaces made of one disc for each braid strand and one band for each generalised crossing.

\begin{figure}[ht]
\scalebox{1.0}{\raisebox{-0pt}{$\vcenter{\hbox{\epsffile{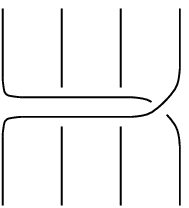}}}$}} 
\caption{}
\end{figure}

A canonical surface is shown at the bottom right of Figure~6, for the braid $\sigma_1 \sigma_{1,3} \sigma_3 \sigma_2 \sigma_3$. In that diagram all bands are represented by forks with two teeth. By looking at the whole sequence of diagrams of Figure~6, we realise that the ribbon surface of any embedded bipartite graph can be deformed into the canonical surface of a strongly quasipositive braid. Note that forks consisting of one edge give no contribution to the canonical surface diagram, since the correponding ribbons can be removed by an isotopy. Conversely, every canonical surface can be deformed into a bipartite graph surface by shrinking each braid strand to a point. In this way we obtain a bipartite graph where all points on the lower line $L$ have valency two. This proves Theorem~1.

\begin{figure}[ht]
\scalebox{0.7}{\raisebox{-0pt}{$\vcenter{\hbox{\epsffile{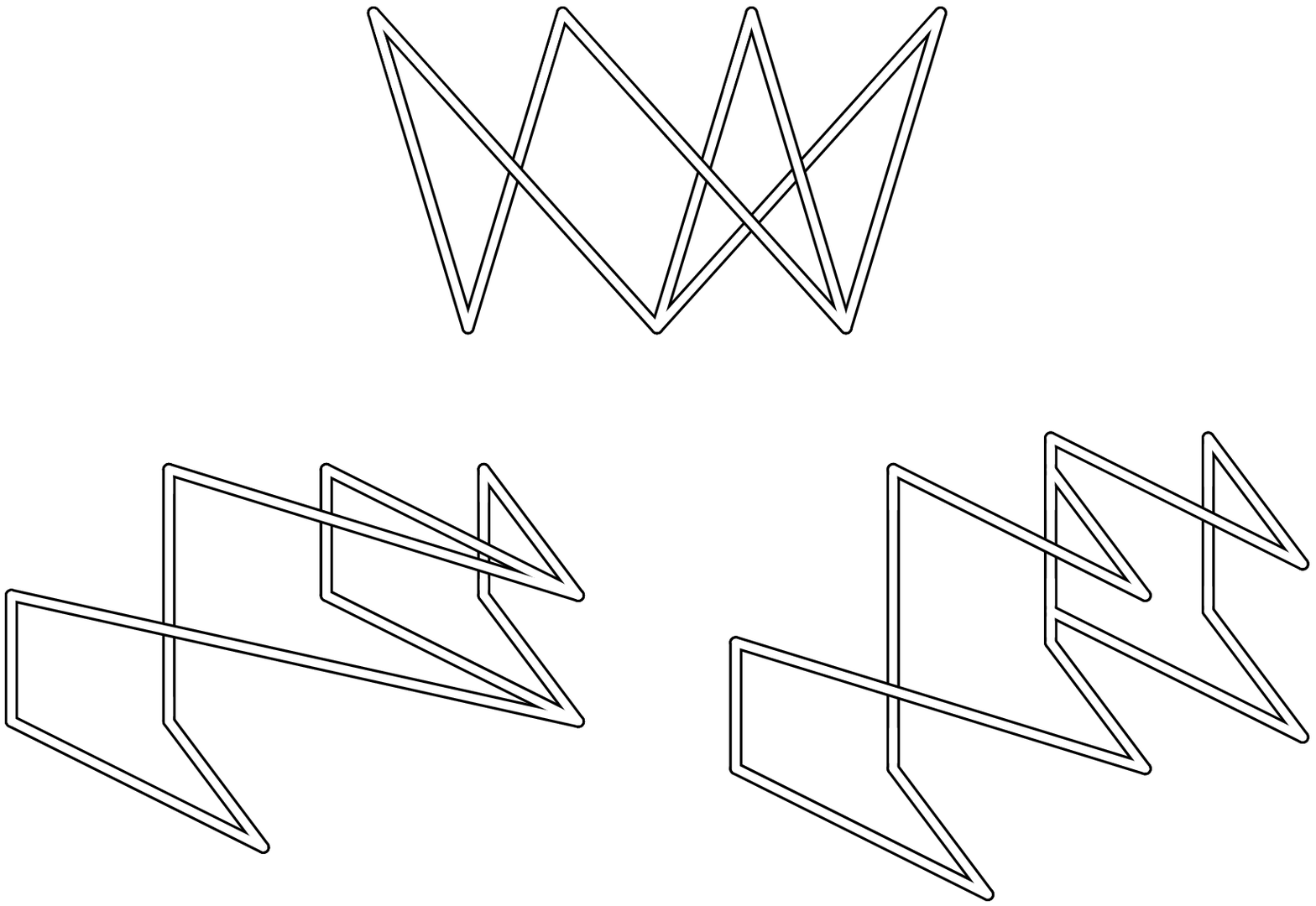}}}$}} 
\caption{}
\end{figure}

\section{Twisted torus links}

As we saw in the preceding section, there is a natural way of deforming the ribbon surface of an embedded bipartite graph $\Gamma$ into the canonical surface of a strongly quasipositive braid. The resulting braid is positive, if and only if all forks of $\Gamma$ are complete, meaning that their endpoints form sets of consecutive vertices of $\Gamma$ on the upper line (as in Figure~3). By turning an embedded bipartite graph $\Gamma$ upside down, we obtain another embedded bipartite graph $\widetilde{\Gamma}$ that gives rise to a different braid, in general. Note that the forks of $\widetilde{\Gamma}$ need not be complete, even if the ones of $\Gamma$ are. Therefore we are not able to extend the duality on torus link diagrams to positive braids. However, we may define a duality on the more restricted class of twisted torus link diagrams. By the work of Birman and Kofman~\cite{BK}, these represent Lorenz links. Twisted torus links admit a natural description in terms of bipartite graphs.

Let $a_1 \geq a_2 \geq \ldots \geq a_n$ be a finite decreasing sequence of natural numbers. We define an embedded bipartite graph $\Gamma(a_1, \ldots, a_n)$ with $a_1$ vertices on the upper line $U$ as a union of $n$ forks, where the $k$-th fork has the first $a_k$ points on $U$ as endpoints (see Figure~7 for an illustration). Braid diagrams associated with graphs of type $\Gamma(a_1, \ldots, a_n)$ are precisely twisted torus link diagrams. Turning $\Gamma(a_1, \ldots, a_n)$ upside down gives rise to a dual graph $\Gamma(b_1, \ldots, b_n)$ with $b_1=n$, $m=a_1$. In fact, thinking of the numbers $a_1, a_2, \ldots, a_n$ as the coefficients of a Young diagram, the numbers $b_1, b_2, \ldots, b_n$ correspond to the dual Young diagram. For example, the dual graph of $\Gamma(4,4,3,2,2)$ is $\Gamma(5,5,3,2)$. The same involution has been described by Birman and Kofman (\cite{BK}, Corollary~4) and, in terms of Lorenz links, by Dehornoy (\cite{De}, Proposition~1.18.). Within the framework of bipartite graphs, it is not hard to figure out a duality on an even larger class of diagrams.

\begin{figure}[ht]
\scalebox{1.0}{\raisebox{-0pt}{$\vcenter{\hbox{\epsffile{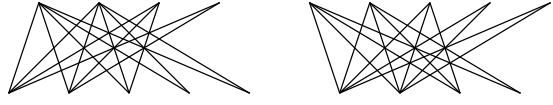}}}$}} 
\caption{$\Gamma(4,4,3,2,2)$ and $\Gamma(5,5,3,2)$}
\end{figure}

\section{Combinatorial adjacency}

Torus links are prototypes of links associated with isolated singularities of complex plane curves. Classically, the torus link $T(p,q)$ is defined as the intersection of the singular curve $\{z^p+w^q=0\}$ with the unit $3$-sphere $\{|z|^2+|w|^2=1\} \subset \C^2$. A generic deformation of that curve will transform it into a smooth one, e.g. $z^p+w^q+t$, $t \in [0,1]$. Carefully chosen deformations may give rise to simpler singularities, e.g. $z^p+w^q+t(z^a+w^b)$, $t \in [0,1]$, where $a \leq p$, $b \leq q$. After a suitable biholomorphic coordinate change around $0 \in \C^2$, the singularity $z^p+w^q+t(z^a+w^b)$ becomes $x^a+y^b$, provided $t>0$. 

Formally, a deformation of a singularity of $f \in \C[z,w]$ at $0 \in \C^2$ is a polynomial in three variables $H(t,z,w)$ with the following properties:
\begin{enumerate}
\item $H(0,z,w)=f(z,w)$;
\item for all $t \in [0,1]$, the restriction $H_t(z,w)=H(t,z,w)$ has an isolated singularity at $0 \in \C^2$;
\item for all $t \in (0,1]$, the singularity of $H_t$ at $0$ is equivalent to the singularity of $H_1$ at $0$ (via a local biholomorphic coordinate change).
\end{enumerate}

We say that the singularity of $H_1$ at $0$ is adjacent to the singularity of $H_0=f$, as well as their links. For more details on versal deformations and adjacency of singularities, we refer the reader to Siersma's dissertation~\cite{Si}. An explicit solution to the adjacency problem is not known, not even for singularities of type $z^p+w^q$. Let us note that the link $T(a,b)$ is adjacent to $T(p,q)$, provided $a \leq p$, $b \leq q$. However, this is not a necessary condition. For example, the links $T(2,n)$ are adjacent to $T(3,4)$, for all $n \leq 6$ (these are the only ones, apart from $T(3,3)$, for genus reasons). More generally, for $a,b,c \in \N$ with $c \leq a$, the function $H(t,z,w)=z^a+(w^b+tz)^c$ exhibits a deformation of the singularity $z^a+w^{bc}$ into $z^{ab}+w^c$. This surprisingly easy deformation was found by Peter Feller, based upon ideas of Ishikawa, Nguyen and Oka~\cite{ITO}. The verification requires basic knowledge about Newton polygons. Substituting $x=w^b+tz$ transforms $z^a+(w^b+tz)^c$ into $(\frac{x-w^b}{t})^a+x^c$; the latter is equivalent to $y^{ab}+x^c$, provided $c \leq a$.   

In this section we propose a combinatorial notion of adjacency for bipartite graph links, motivated by the above list of algebraic adjacencies. A complete bipartite graph of type $\theta_{p,q}$ consists of $p$ forks with $q$ teeth. Given a fork $f$ and two natural numbers $a_1, a_2 \geq 1$ with $a_1+a_2=q$, let us define a splitting of $f$ by assembling the first $a_1$ and the last $a_2$ endpoints of $f$ into two forks $f_1$, $f_2$, respectively (see Figure~8). On the level of braid diagrams, this is simply smoothing a crossing. We say that an embedded bipartite graph $\Gamma$ is adjacent to $\theta_{p,q}$, if it is obtained from $\theta_{p,q}$ by a finite number of fork splittings, where splittings may be applied to up- and down-pointing forks, possibly with iteration (see Figure~2).
The corresponding link $L(\Gamma)$ is called combinatorially adjacent to $T(p,q)$ 
\footnote{This definition arose from discussions with Peter Feller}.
As a first example, we note that the torus link $T(a,b)$ is combinatorially adjacent to $T(p,q)$, provided $a \leq p$, $b \leq q$ (this is most easily seen in two steps, via $T(a,q)$ resp. $T(p,b)$). A more interesting splitting is shown in Figure~8; it shows that $T(2,6)$ is combinatorially adjacent to $T(3,4)$. 

\begin{figure}[ht]
\scalebox{0.7}{\raisebox{-0pt}{$\vcenter{\hbox{\epsffile{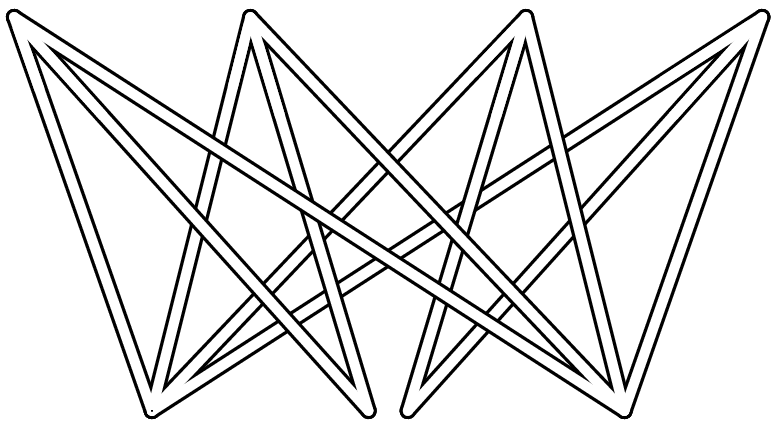}}}$}} 
\caption{}
\end{figure}

This is a special case of the following fact: for $a,b,c \in \N$, $c \leq a$, the torus link $T(ab,c)$ is obtained from $T(a,bc)$ by smoothing an appropriate set of crossings, $(b-1)(a-c)$ in number (see~\cite{Ba}, Proposition~1, and Figure~9 for $(a,b,c)=(3,2,7)$ and $(2,3,7)$). This is the combinatorial counterpart to the above statement on algebraic adjacencies. 

\begin{figure}[ht]
\scalebox{0.7}{\raisebox{-0pt}{$\vcenter{\hbox{\epsffile{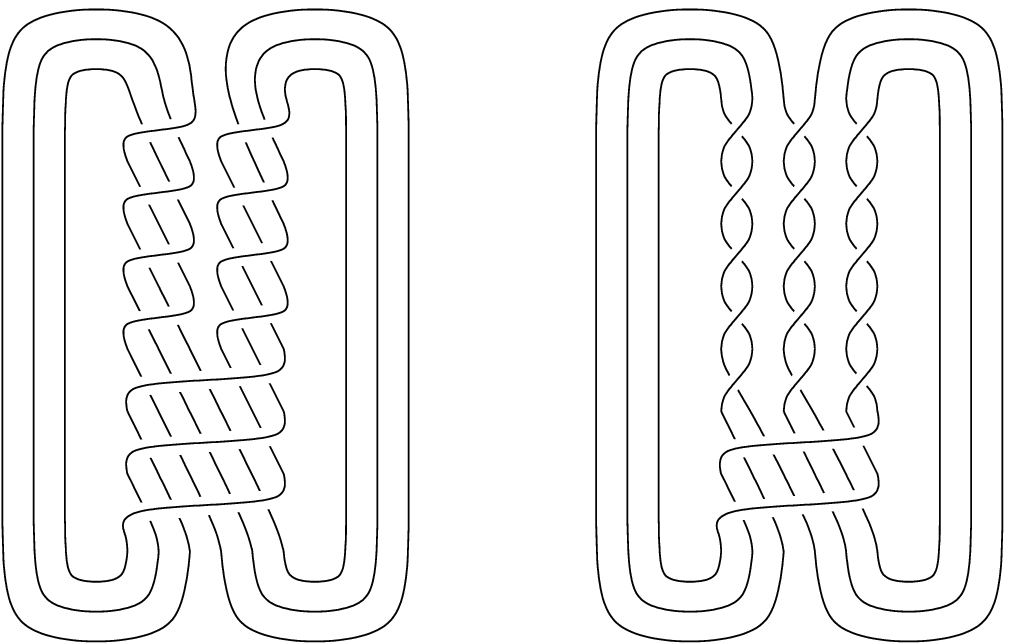}}}$}} 
\caption{}
\end{figure}

The exciting thing about combinatorial adjacency is that it allows smoothings of crossings in the two standard diagrams of torus links simultaneously. As we just observed, there is a big overlap between algebraic and combinatorial adjacency, where both notions apply. We do not know whether this is more than a pure coincidence. The following question should nevertheless be justified: for which $a,b,p,q \in \N$ is $T(a,b)$ combinatorially adjacent to $T(p,q)$?

\section{Density of bipartite graph links}

Let $\Gamma \subset \theta_{p,q}$ be an embedded bipartite graph with $p+q$ vertices. We define the density of $\Gamma$ as
$$d(\Gamma)=\frac{e(\Gamma)}{pq},$$
where $e(\Gamma)$ is the number of edges of $\Gamma$. This carries over to bipartite graph links $L$, by taking the supremum over all bipartite graph representatives for $L$:
$$d(L)=\sup_{L=L(\Gamma)} d(\Gamma).$$

\begin{proposition} \quad
\begin{enumerate}
\item[(1)] The supremum is actually a maximum, i.e. there exists an embedded bipartite graph $\Gamma$ with $d(\Gamma)=d(L)$,\\
\item[(2)] $d(K) \leq \frac{2}{b}-\frac{\chi}{b^2}$, for all bipartite graph knots $K$, where $b$ and $\chi$ denote the braid index and the maximal Euler characteristic of the knot $K$, respectively. \\
\end{enumerate}
\end{proposition}

\begin{corollary}
$d(L)=1 \Leftrightarrow L$ is a torus link.
\end{corollary}

\begin{proof}[Proof of Proposition~1] 
Let $\Gamma$ be an embedded bipartite graph representative for $L$. The corresponding ribbon surface $F=F(\Gamma)$ is a Seifert surface of maximal Euler characteristic for the link $L$, by Bennequin's inequality~\cite{Be}: $\chi(F)=\chi(L)$. In terms of $\Gamma$, we compute
$$\chi(F)=p+q-e(\Gamma)=p+q-pqd(\Gamma),$$
thus
$$d(\Gamma)=\frac{p+q-\chi(L)}{pq}.$$
This implies the second item of the Proposition, since $p,q \geq b$ and $\chi$ is non-positive for knots (except for the trivial knot, for which the
statement is true anyway). Choose $N \in \N$ so that $N \geq |\chi(L)|$ and $N>\frac{3}{d(L)}$. If $p,q \geq N$ then
$$d(\Gamma)=\frac{1}{q}+\frac{1}{p}+\frac{-\chi(L)}{pq} \leq \frac{3}{N}<d(L).$$
Therefore, in order to come close to $d(L)$, the graph $\Gamma$ must satisfy $p,q \leq N$. This implies the first item, since there are only finitely many
embedded bipartite graphs with $p,q \leq N$.
\end{proof}

A link is called special alternating, if it has a diagram which is both positive and alternating~\cite{Mu}. Since positive links are quasipositive~\cite{Ru2}, it makes sense to talk about their density.

\begin{proposition} Let $K$ be a special alternating knot of braid index at least $3$.
Then either $K=T(3,4)$ or $d(K) \leq \frac{24}{25}$.
\end{proposition}

The restriction to knots of braid index at least $3$ is only for the sake of simplicity of the statement;
special alternating knots of braid index $2$ are torus links of type $T(2,n)$, $n \geq 3$.

\begin{proof}[Proof of Proposition~2]
Let $K=L(\Gamma)$ be represented by an embedded bipartite graph $\gamma \subset \theta_{p,q}$. If $K \neq T(3,4)$ then $d(K) \neq 1$, since the only alternating torus knots are $T(2,n)$ , $n \in \N$, and $T(3,4)$. Suppose, for contradiction, that $\frac{24}{25}<d(K)<1$. Then the graph $\Gamma$ has more than $25$ edges (in particular, one of the numbers $p,q$ is at least $6$). Moreover, $\Gamma$ contains a complete bipartite graph of type $\theta_{3,6}$ as a subgraph. Indeed, if no $(3,6)$-bipartite subgraph of $\Gamma$ was complete, then $d(\Gamma) \leq \frac{17}{18}<\frac{24}{25}$. In terms of Seifert surfaces, this means that the ribbon surface $F(\Gamma)$ associated with $\Gamma$ contains the fibre surface of the torus link $T(3,6)$ as an incompressible subsurface. A direct computation, carried out in~\cite{GLM}, shows that the symmetrised Seifert matrix of the fibre surface of $T(3,6)$ has non-maximal signature ($\sigma(T(3,6))=8<10$). This fact descends to the surface $F(\Gamma)$, since maximality of the signature is preserved under taking minors of symmetric matrices. We conclude
$$\sigma(K)<2g(K),$$
where $g(K)$ denotes the genus of the knot $K$. On the other hand, since $K$ is positive, its Rasmussen invariant $s(K)$ coincides with twice the genus: $s(K)=2g(K)$. But $K$ is also alternating, so $\sigma(K)=s(K)$, a contradiction (see~\cite{Ra} for the last two facts).
\end{proof}

A careful generalisation of the above proof shows that a large density implies $\frac{\sigma}{2g} \leq \frac{3}{4}$. 
The reader might suspect there is a kind of converse, i.e. a small density implies $\frac{\sigma}{2g} \approx 1$. However, this is not true, as shows the following example.

\begin{example}
Let $n \geq 5$ and let $\Gamma_n \subset \theta_{n,n+1}$ be the embedded bipartite subgraph obtained from the complete bipartite graph $\theta_{4,5}$ by the extension shown in Figure~10, for $n=9$. We observe that $K_n=L(\Gamma_n)$ is
a knot of genus $g(K_n)=g(T(4,5))=6$. A direct computation yields $\sigma(K_n) \leq 10$, thus $\frac{\sigma_{K_n}}{2g(K_n)} \leq \frac{5}{6}$ (this is due to the fact that the torus link $T(4,4)$ has non-maximal signature). Moreover, the braid index of $K_n$ grows linearly in $n$, similarly as for twist knots with increasing twist
number. This follows for example from the inequality of Morton and Franks-Williams~\cite{Mo},~\cite{FW}. The inequality of Proposition~1 implies
$$\lim_{n \to \infty} d(K_n)=0.$$

Therefore a small density does not necessarily imply $\frac{\sigma}{2g} \approx 1$.

\begin{figure}[ht]
\scalebox{1.0}{\raisebox{-0pt}{$\vcenter{\hbox{\epsffile{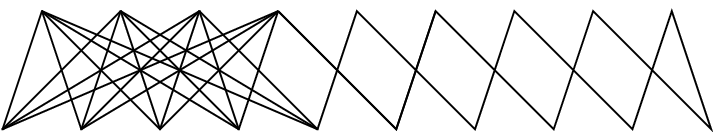}}}$}} 
\caption{}
\end{figure}

\end{example}

The proof of Proposition~2 suggests to study bipartite graph links via graph minor theory. For example, we may ask if there exist finitely many embedded bipartite graphs $\Gamma_1, \ldots, \Gamma_n \subset \R^3$ such that the following holds: the signature of a bipartite graph knot $K=L(\Gamma)$ is maximal, $\sigma(K)=2g(K)$, if and only if $\Gamma$ contains none of the graphs $\Gamma_i$ as embedded minors.

\bigskip
\noindent
Universit\"at Bern, Sidlerstrasse 5, CH-3012 Bern, Switzerland

\bigskip
\noindent
\texttt{sebastian.baader@math.unibe.ch}

\end{document}